\documentclass[11pt]{amsart}
\usepackage{amssymb,amsmath}

\makeatletter
\theoremstyle{plain}
 \newtheorem{thm}{Theorem}[section]
\newtheorem{thm*}{Theorem}
 
 \newtheorem{prop}[thm]{Proposition}
 
 \numberwithin{equation}{section} 
\numberwithin{figure}{section} 
 \theoremstyle{plain}
 \theoremstyle{definition}

\newcommand{\K}{{{\mathbb K}}}

\newcommand{\mH}{{{\mathbb H}}}

\newcommand{\cJ}{{{\mathcal J}}}

\newcommand{\C}{{{\mathbb C}}}
\newcommand{\R}{{{\mathbb R}}}

\newcommand{\bH}{{{\bf H}}}
\newcommand{\mO}{{{\mathbb O}}}
\newcommand{\bp}{{{\bf p}}}
\newcommand{\bv}{{{\bf v}}}
\newcommand{\bw}{{{\bf w}}}

\newcommand{\be}{{{\bf e}}}

\newcommand{\cR}{{{\mathcal R}}}
\newcommand{\fp}{{{\mathfrak p}}}

\newcommand{\mP}{{{\mathbb P}}}

\newcommand{\X}{{{\mathbb X}}}

\newcommand{\fH}{{{\mathfrak H}}}

\newcommand{\binfty}{{{\bf \infty}}}

\makeatother

\begin{document}

\title[Cross--ratios and Ptolemaean Inequality]{Cross--Ratios and the Ptolemaean Inequality in boundaries of  symmetric spaces of rank 1}

\author[I.D. Platis]{Ioannis D. Platis}

\begin{abstract}
We use generalised  cross--ratios to prove the Ptolemaean inequality and the Theorem of Ptolemaeus in the setting of the boundary of  symmetric Riemannian spaces of rank 1 and of negative curvature.
\end{abstract}

\address{Department of Mathematics,
University of Crete, Knossos Ave., GR 71409, 
Heraklion Crete, Greece}

\email{jplatis@math.uoc.gr}

\maketitle

\centerline{\it In memoriam patris mei: Demos I. Platis, 1926--2002.}

\section{Introduction}

The Theorem of Ptolemaeus in planar Euclidean geometry states that the product of the euclidean lengths of the diagonals of an inscribed quadrilateral equals to the sum of the products of the euclidean lengths of its opposite sides. When one vertex of the quadrilateral does not lie on the circle passing from the other three vertices, then we have inequality, known as the Ptolemaean inequality.

The intrinsic significance of Ptolemaean inequality was already known in antiquity.  In the modern era and specifically in the time period covering at least the past sixty years, its generalisation to various spaces has been the study of many authors: illustratively, see the old paper of  Schoenberg  for a generalisation into normed spaces \cite{S}, the  work of Buckley, Falk and Wraith in CAT(0) spaces \cite{BFW}, and   the paper of   Buyalo and Schroeder for a more  general setting in abstract spaces \cite{BS}.

In the present paper we give an elementary  proof of the Ptolemaean Inequality and the Theorem of Ptolemaeus in the boundary of  symmetric Riemannian spaces  of rank 1 and of negative curvarure. These spaces are the $n-$dimensional hyperbolic spaces $\bH^n_\K$ where $\K$ can be one of the following: i)  the set of real numbers $\R$, ii) the set of complex numbers $\C$, iii) the set of quaternions $\mH$ and iv)  the set of octonions $\mO$ (in this case $n=2$).  A Riemannian metric of negative sectional curvature is defined in spaces $\bH^n_\K$ and let $G_\K$ be the isometry group of this metric. It is well known that $G_\R={\rm SO}(n,1)$, $G_\C={\rm SU}(n,1)$, $G_\mH={\rm Sp}(n,1)$ and $G_\mO=F_{4(-20)}$. The action of $G_\K$ is extended to an action in the boundary $\partial\bH^n_\K$; the boundary is a sphere, and after applying stereographic projection we obtain a set which can be endowed with a natural structure of a Lie group $\fH_\K$. This group is the additive group $\R^{n-1}$ when $\K=\R$ and the 
generalised Heisenberg group $\K^{n-1}\times\Im(\K)$ in all other cases. A metric $d_\fH$ is defined in $\fH_\K$ and extended in a natural way to $\partial\bH^n_\K$; it is proved that its similarity group together with an inversion produce the whole isometry group $G_\K$. Note that in the case where $\K=\R$, $d_\fH$ is just the Euclidean metric in $\R^{n-1}$ and this is the only case where $d_\fH$ is a path metric.  

Let now $(X,d)$ be a metric space. The metric $d$ is called {\it Ptolemaean}  if any four distinct points $p_1$, $p_2$, $p_3$ and $p_4$ in $X$ satisfy the Ptolemaean Inequality: for any permutation $(i,j,k,l)$ in the permutation group $S_4$  we have
$$
d(p_i,p_k)\cdot d(p_j,p_l)\le d(p_i,p_j)\cdot d(p_k,p_l)+d(p_j,p_k)\cdot d(p_k,p_i).
$$
A subset $\sigma$ of  $X$ is called a {\it Ptolemaean circle} if the Theorem of Ptolemaus holds in $\sigma$. That is,  for any four distinct points $p_1$, $p_2$, $p_3$ and $p_4$ in $\sigma$ such that $p_1$ and $p_3$ separate $p_2$ and $p_4$ we have
$$
d(p_1,p_3)\cdot d(p_2,p_4)=d(p_1,p_2)\cdot d(p_3,p_4)+d(p_2,p_3)\cdot d(p_4,p_1).
 $$ 
Our main result is the following.

\medskip

{\bf Theorem.} {\it The metric $d_\fH$ in $\partial\bH^n_\K$ (or in $\fH_\K$) is Ptolemaean. Its Ptolemaean circles are $\R-$circles.}

\medskip

Its proof follows from Theorems \ref{thm:ptolemaean-ineq}, \ref{thm:ptolemaean-eq}, \ref{thm:ptolemaean-heis}, \ref{thm:ptolemaean-ineq-oct}, \ref{thm:ptolemaean-eq-oct} and  \ref{thm:ptolemaean-heis-oct}; for the definition of $\R-$circles in the boundary of $\bH^n_\K$ see Sections \ref{sec:K-Rcircles} and \ref{sec:R-sirc-O}. We only note here that in the case where $\K=\R$, $\R-$circles are the familiar Circles, i.e. Euclidean circles or straight lines. The picture in every other case is quite different.

The proof of our main result relies entirely on cross--ratios and their properties; we give a unifying exposition, distinguishing only the octonionic case. In \cite{P-cr} we developed basically the same method to prove the result for the case where $\K=\C$ and $n=2$ using Kor\'anyi--Reimann complex cross--ratios. Here we define $\K-$cross--ratios for the cases when $\K=\R,\C,\mH$ and real cross-ratios for the case where $\K=\mO$. Given a quadruple $\fp=(p_1,p_2,p_3,p_4)$ of distinct points in the boundary of $\bH^n_\K$ ($\K=\R,\C,\mH$), the cross--ratio $\X(p_1,p_2,p_3,p_4)$ is a map from $\left(\bH^n_\K\right)^4\setminus\{\text{diagonals}\}\to\K$ which is invariant by the action of $G_\K$ and satisfies certain symmetry conditions, see Section \ref{sec:cr} for details. We only note here that in the familiar case of $\K=\R$ and $n=2$ the $\R-$cross--ratio of $\fp$ is the square of the well known projective invariant
$$
[p_1,p_2:p_3,p_4]=\frac{|p_4-p_2|}{|p_4-p_1|}:\frac{|p_3-p_2|}{|p_3-p_1|}.
$$
Our $\K-$cross--ratio $\X(p_1,p_2,p_3,p_4)$ is the natural genaralisation of the above invariant. One of its properties that we use is that its modulus satisfies
$$
|\X(p_1,p_2,p_3,p_4)|^{1/2}=\frac{d_\fH(p_4,p_2)}{d_\fH(p_4,p_1)}:\frac{d_\fH(p_3,p_2)}{d_\fH(p_3,p_1)},
$$
and in the case where  $\K=\R$ we have $\X=|\X|$.
Of course, real cross--ratios are classic and complex cross--ratios ($\K=\C$, $n=2$) introduced by Kor\'anyi and Reimann in \cite{KR1} have been studied quite extensively, see for instance \cite{F}, \cite{FP} and \cite{PP}. Treatises on quaternionic cross--ratios are found in  \cite{BG} as well as in \cite{GL} but only in the case where $n=1$. 
It seems that there is no natural way to define an octonionic cross--ratio endowed with all the properties shared by its real, complex and quaternionic counterpart.  An attempt to mimick our construction in the case where $\K=\R,\C,\mH$,  or the construction in \cite{KR1}, will fail due to non associativity of the octonionic multiplication. However, there is a  real cross--ratio defined in the boundary of octonionic hyperbolic plane which is sufficient for our purposes, see Section \ref{sec:real-cr-oct}. 

In all cases, we prove that certain cross--ratios associated to a quadruple of four points in the boundary satisfy two fundamental relations, see Propositions \ref{prop-norm} and \ref{prop:X-oct}. We then obtain our main result by exploiting these relations. 

The paper is organised as follows. In Section \ref{sec:prel1} we state the basics for  complex hyperbolic space and its boundary when $\K=\R,\C,\mH$. In Section \ref{sec:cr} we introduce $\K-$cross--ratios and prove the fundamental relations which lead to our main result in this case; the latter is proved in Section \ref{sec:pt1}.  In Section \ref{sec:prelO} we describe the octonionic hyperbolic space and its boundary and finally, we prove our main result for the octonionic case in Section \ref{sec:pt2}.

\subsubsection*{Aknowledgement} The author would like to thank Zolt\'an M. Balogh and John R. Parker for some useful discussions and suggestions.  

\section{Preliminaries for the Case $\K=\R, \C, \mH$}\label{sec:prel1}

The material in this section is quite standard; we refer the reader to the books of Mostow  \cite{M}, Goldman \cite{Gol}, to the paper of Kim and Parker for the quaternionic case \cite{KP}, and to the notes of Parker \cite{P} for the complex 2--dimensional case.
In the spirit of Mostow,  we define (the Siegel domain model for) $\K-$hyperbolic space $\bH^n_\K$ in Section \ref{sec:K-hyp},  and its isometries in \ref{sec:K-isom}. In Section \ref{sec:K-H} we treat the boundary $\partial\bH^n_\K$, the induced group $\fH_\K$, the metric $d_\fH$ defined on this group, its isometries and similarities. $\R-$circles are presented in the separate Section \ref{sec:K-Rcircles}.

\subsection{$\K-$Hyperbolic Space}\label{sec:K-hyp}

Let $\K$ be $\R$, $\C$, $\mH$ and let $\K^{n,1}$ be the vector space $\K^{n+1}$ with
the Hermitian form of signature $(n,1)$ given by
$$
\left\langle {\bf {z}},{\bf {w}}\right\rangle 
={\bf w}^*J{\bf z}
=\overline{w}_{n+1}z_{1}+\overline{w}_{n}z_{n}+\dots \overline{w}_{2}z_{2}+\overline{w}_{1}z_{n+1}
$$
with matrix 
$$
J=\left[\begin{array}{ccc}
0 & 0 & 1\\
0 & I & 0\\
1 & 0 & 0\end{array}\right],
$$
where $I$ is the $(n-1)\times (n-1)$ identity matrix. The order of multplication in the form does not make a difference unless $\K=\mH$.
We consider the following subspaces of ${\K}^{n,1}$:
\begin{eqnarray*}
V_- & = & \Bigl\{{\bf z}\in\K^{n,1}\ :\ 
\langle{\bf z},\,{\bf z} \rangle<0\Bigr\}, \\
V_0 & = & \Bigl\{{\bf z}\in\K^{n,1}-\{{\bf 0}\}\ :\ 
\langle{\bf z},\,{\bf z} \rangle=0\Bigr\}.
\end{eqnarray*}
Let ${\mathbb P}:\K^{n,1}\setminus\{{\bf 0}\}\longrightarrow \K P^n$ be
the canonical projection onto $\K-$projective space; we choose 
the right projection  for the case $\K=\mH$. Then 
{\sl $\K-$hyperbolic space} ${\bf H}_{\K}^{n}$
is defined to be ${\mathbb P}V_-$ and its boundary
$\partial{\bf H}^n_{\K}$ is ${\mathbb P}V_0$.
Specifically, $\K^{n,1}\setminus\{{\bf 0}\}$ may be covered with $n+1$ charts
$H_1,\dots,H_n$ where $H_j$ comprises those points in 
$\K^{n,1}\setminus\{{\bf 0}\}$ for which $z_j\neq 0$. It is clear that
$V_-$ is contained in $H_{n+1}$. The canonical projection
from $H_{n+1}$ to $\K^n$ is given by 
${\mathbb P}({\bf z})=(z_1z_{n+1}^{-1},\,z_2z_{n+1}^{-1},\dots,z_nz_{n+1}^{-1})=z$. Therefore we can write
${\bf H}^n_{\K}={\mathbb P}(V_-)$ as
$$
{\bf H}^n_{\K} = \left\{ (z_1,\,z_2,\dots,\,z_n)\in{\K}^n
\ : \ 2\Re(z_1)+\sum_{i=2}^n|z_i|^2<0\right\},
$$
which is called the {\it Siegel domain model for} ${\bf H}^n_{\K}$; see \cite{Gol} for the case $\K=\C$  and \cite{M} for the general case. In the latter, $\K-$hyperbolic space is introduced via the {\it ball model}, that is
$$
{\bf H}^n_{\K} = \left\{ (z_1,\,z_2,\dots,\,z_n)\in{\K}^n
\ : \ \sum_{i=1}^n|z_i|^2<1\right\},
$$
but it turns out that these two definitions are equivalent.
There are distinguished points in $V_0$ which we denote by 
${\bf o}$ and $\binfty$:
$$
{\bf o}=\left[\begin{matrix} 0 \\ 0\\ \vdots\\0\\  1 \end{matrix}\right], \quad
\binfty=\left[\begin{matrix} 1 \\0\\ \vdots \\0 \\ 0 \end{matrix}\right].
$$
Then $V_0\setminus\{{\binfty}\}$ is contained in $H_{n+1}$ and 
$V_0\setminus\{{\bf o}\}$ (in particular $\infty$) is contained
in $H_1$. Let ${\mathbb P}{\bf o}=o$ and ${\mathbb P}\binfty=\infty$. 
Then we can write $\partial{\bf H}^n_{\K}={\mathbb P}(V_0)$ as
$$
\partial{\bf H}^n_{\K}\setminus\{\infty\} 
=\left\{ (z_1,\,z_2,\dots,\,z_n)\in{\K}^n
\ : \ 2\Re(z_1)+\sum_{i=2}^n|z_i|^2=0\right\}.
$$
In particular $o=(0,\dots,\,0)\in\K^n$. 

Conversely, given a point $z$ of 
${\K}^n={\mathbb P}(H_{n+1})\subset\K P^n$ we may 
lift $z=(z_1,\dots,\,z_n)$ to a point ${\bf z}$ in $H_{n+1}\subset\K^{n,1}$, 
called the {\sl standard lift} of $z$, by writing ${\bf z}$ in non-homogeneous
coordinates as
$$
{\bf z}=\left[\begin{matrix} z_1 \\ z_2\\ \vdots\\z_n \\ 1 \end{matrix}\right].
$$

The {\it Riemannian metric} on  ${\bf H}_{\K}^{n}$ is defined by the
distance function $\rho$ given by the formula
$$
\cosh^{2}\left(\frac{\rho(z,w)}{2}\right)
=\frac{\left\langle {\bf {z}},{\bf {w}}\right\rangle 
\left\langle {\bf {w}},{\bf {z}}\right\rangle }
{\left\langle {\bf {z}},{\bf {z}}\right\rangle 
\left\langle {\bf {w}},{\bf {w}}\right\rangle }
=\frac{\bigl|\langle {\bf z},{\bf w}\rangle\bigr|^2}
{|{\bf z}|^2|{\bf w}|^2}
$$
where ${\bf z}$ and ${\bf w}$ in $V_-$ are the standard lifts of $z$ and $w$ 
in ${\bf H}^n_{\K}$ and
$|{\bf z}|=\sqrt{-\langle{\bf z},{\bf z}\rangle}$.
Alternatively,
$$
ds^{2}=-\frac{4}{\left\langle {\bf {z}},{\bf {z}}\right\rangle ^{2}}
\det\left[\begin{array}{cc}
\left\langle {\bf {z}},{\bf {z}}\right\rangle  
& \left\langle d{\bf {z}},{\bf {z}}\right\rangle \\
\left\langle {\bf {z}},d{\bf {z}}\right\rangle  
& \left\langle d{\bf {z}},d{\bf {z}}\right\rangle \end{array}\right].
$$
The real sectional curvature of 
$ {\bf H}_{\K}^{n}$ is $-1/4$  when $\K=\R$ and when $\K=\C$ or $\mH$ it is pinched between $-1$ and $-1/4$.
Also, when $\K=\C$ then $\bH^n_\C$ is a complex manifold, the metric is K\"ahler (in fact, it is the {\it Bergman metric}) and  the holomorphic sectional curvature equals to $-1$.

\subsubsection{Isometries, $\K-$lines, $\R-$planes}\label{sec:K-isom}

Denote by ${\rm F}(n,1)$ the group
of unitary matrices for  the  Hermitian form 
$\left\langle \cdot,\cdot\right\rangle $. This is
\begin{enumerate}
 \item the group ${\rm O}(n,1)$ when $\K=\R$;
\item the group ${\rm U}(n,1)$ when $\K=\C$ and
\item the group ${\rm Sp}(n,1)$ when $\K=\mH$.
\end{enumerate}
Each  matrix $A\in{\rm F}(n,1) $ satisfies
the relation $A^{-1}=JA^{*}J$ where $A^{*}$ is the Hermitian transpose of $A$.
The isometry group $G_\K$ of 
$\K-$hyperbolic space is  the {\it projective group} ${\rm PF(n,1)}$. Instead, we may use the following groups for $G_\K$:
\begin{enumerate}
 \item ${\rm SO}(n,1)$ (a double cover of ${\rm PO}(n,1)$) when $\K=\R$;
\item ${\rm SU}(n,1)$ (a triple cover of ${\rm PU}(n,1)$) when $\K=\C$ and
\item ${\rm Sp}(n,1)$  (a double cover of ${\rm PSp}(n,1)$) when $\K=\mH$.
\end{enumerate}
The latter is consistent with the fact that $\bH^n_\K$ is a Riemannian symmetric space, i.e. it  is i) ${\rm SO}(n,1)/{\rm SO}(n)$ when $\K=\R$, ii) ${\rm SU}(n,1)/{\rm U}(n)$ when $\K=\C$ and iii) ${\rm Sp}(n,1)/{\rm Sp}(n)$ when $\K=\mH$. 

We shall also denote by ${\rm F}(n)$  the group of isometries of the usual (Euclidean) hermitian product in $\K^n$; that is ${\rm O}(n)$ when $\K=\R$, ${\rm U}(n)$ when $\K=\C$ and ${\rm Sp}(n)$ when $\K=\mH$ respectively.

Two kinds of subspaces of $\bH^n_\K$ are of our special interest, that is $\K-$lines and mainly $\R-$planes. For details in the complex case see \cite{Gol}, the other cases are treated similarly.

A $\K-${\it line} is an isometric image of the embedding of $\bH^1_\K=\{z\in\K\;|\;\Re(z)<0\}$ into $\bH^n_\K$. We may assume that the embedding is the standard one
$$
z\mapsto(z,0,\dots,0).
$$
The isometries preserving a $\K-$line is a subgroup of $G_\K$ isomorphic to ${\rm F}(1,1)$. 

An $\R-${\it plane} $\cR$ is a real 2-dimensional subspace of $\bH^n_\K$ characterised by $\langle \bv,\bw\rangle\in\R$ for all $\bv,\bw\in\cR$ (the latter is of course vacuous when $\K=\R$). Any real plane $\cR$ is  the isometric image 
of an embedded copy of $\bH_\R^2=\{(x_1,x_2)\in\R^2\;|\;2x_1+x_2^2<0\}$ into $\bH^n_\K$; here, we may assume that the embedding is the standard one
$$
(x_1,x_2)\mapsto(x_1,x_2,0,\dots,0,0,\dots,0).
$$
The isometries preserving the plane above is a subgroup of $G_\K$ isomorphic to ${\rm PO}(2,1)$.

\subsection{The boundary and the group $\fH_\K$}\label{sec:K-H}

A finite point $z$ is in the boundary of the Siegel domain if its standard
lift to $\K^{n,1}$ is ${\bf z}$ where
$$
{\bf z}=\left[\begin{matrix} z_1 \\ \vdots \\z_n \\ 1 \end{matrix}\right]
\quad \text{ where }\quad 2\Re(z_1) +\sum_{i=2}^n |z_i|^2 = 0.
$$
We write $$
\zeta_i=z_{i+1}/\sqrt{2},\;i=1,\dots, n-1,\quad \zeta=(\zeta_1,\dots,\zeta_{n-1})\in\K^{n-1},
$$ 
and this 
condition becomes $2\Re(z_1)=-2\sum_{i=1}^{n-1}|\zeta_i|^2=-2\|\zeta\|^2$, where $\|\cdot\|$ is the Euclidean norm in $\K^{n-1}$. Hence we may
write $z_1=-\|\zeta\|^2+v$ where 
\begin{enumerate}
 \item $v=0$ if $\K=\R$;
\item  $v\in\Im(\C)$, i.e $v=it$, $t\in\R$ if $\K=\C$ and
\item $v\in\Im(\mH)$, i.e. it is a purely imaginary quaternion if $\K=\mH$.
\end{enumerate}

Therefore
$$
{\bf z}
=\left[\begin{matrix} -\sum_{i=2}^n |\zeta_i|^2+v \\ \sqrt{2}\zeta_1\\ \vdots\\ \sqrt{2}\zeta_{n-1} \\1\end{matrix}\right]=
\left[\begin{matrix} -\|\zeta\|^2+v \\ \sqrt{2}\zeta \\1\end{matrix}\right].
$$
In this way, and for $n>1$, we may identify the boundary of the Siegel domain with the one point compactification of $\K^{n-1}\times\Im(\K)$, that is
\begin{enumerate}
 \item the sphere $S^{n-1}$ considered as the one point compactification of $\R^{n-1}\times\{\;0\}\simeq \R^{n-1}$ if $\K=\R$
\item the sphere $S^{2n-1}$ considered as the one point compactification of $\C^{n-1}\times\Im(\C)\simeq\C^{n-1}\times\R$ if $\K=\C$ and
\item the sphere $S^{4n-1}$ considered as the one point compactification of $\mH^{n-1}\times\Im(\mH)$ if $\K=\mH$.
\end{enumerate}
In the exceptional case $n=1$, the boundary of the Siegel domain is a single point when $\K=\R$; henceforth we shall not deal with this case.

The action of the stabiliser of infinity ${\rm Stab}(\infty)$ gives to the set of these points the structure of a  group which we shall denote $\fH_\K$. The  group law is
$$
(\zeta,v)*(\zeta',v')=\left(\zeta+\zeta',v+v'+2\omega(\zeta,\zeta')\right)
$$
where $\omega$ is the standard symplectic product in $\K^{n-1}$ (for $\K=\R$ this is identically zero). More explicitly,
$$
(\zeta_1,\dots,\zeta_{n-1},v)*(\zeta_1',\dots,\zeta_{n-1}',v')=(\zeta_1+\zeta_1',\dots,\zeta_{n-1}+\zeta_{n-1}',v+v'+2\sum_{i=1}^{n-1}\Im(\overline{\zeta_i'}\zeta_i)).
$$
In this manner, $\fH_\K$ is
\begin{enumerate}
 \item The additive group $\R^{n-1}\times\{0\}$ if $\K=\R$.
\item The $(n-1)-$Heisenberg group $\C^{n-1}\times\R$ if $\K=\C$ and $n>1$. If $n=1$ it is isomorphic to the additive group $\R$.
\item The $(n-1)-$quaternionic Heisenberg group $\mH^{n-1}\times\Im(\mH)$ if $\K=\mH$ and $n>1$. If $n=1$ it is isomorphic to the additive group $\R^3$.
\end{enumerate}

There is a  {\it gauge} $|\cdot|_\K$ defined on $\fH_\K$  and given by
$$
\left|(\zeta,v)\right|_\K=\left|(\zeta_1,\dots,\zeta_{n-1},v)\right|_\K=\left| \sum_{i=1}^{n-1}|\zeta_i|^2+v\right|^{1/2}=\left|-\|\zeta\|^2+v\right|^{1/2},
$$
where on the right hand side we have the Euclidean norm. Observe that in the cases of $\K=\R$, the  gauge is just the Euclidean norm and the same holds for $\K\neq\R$, $n=1$. In all other cases $|\cdot|_\K$ is not a norm in the usual sense. (In the complex case it is known as the {\it Kor\'anyi--Cygan} gauge. However, from this gauge we obtain  a metric on $\fH$, which we shall denote by $d_\fH$ and is defined by the relation
$$
d_\fH\left((\zeta,v),\,((\zeta',v')\right)
=\left|((\zeta,v)^{-1}*((\zeta',v')\right|_\K.
$$
This metric is not a path metric unless $\K=\R$ or $\K\neq\R$, $n=1$ (in this case $d_\R$ is just the Euclidean metric).
By taking the standard lift of points on $\partial{\bf H}^n_\K\setminus\{\infty\}$ 
to $\K^{n,1}$ we can write the  metric $d_\fH$ as:
$$
d_\fH\left((\zeta,v),\,(\zeta',v')\right)
=\left|\left\langle\left[ \begin{matrix}
-\|\zeta\|^2+v \\ \sqrt{2}\zeta\\1 \end{matrix}\right],\,
\left[ \begin{matrix}
-\|\zeta'\|^2+v' \\  \sqrt{2}\zeta' \\ 1 \end{matrix}\right]
\right\rangle\right|^{1/2}.
$$
The  metric $d_\fH$ is invariant under the following transformations.
\begin{enumerate}
 \item Left translations: given a point $(\zeta',v')\in\fH_\K$ we define
$$
T_{(\zeta',v')}(\zeta,v)=(\zeta',v')*(\zeta,v).
$$
Left translations are essentially the left action of $\fH_\K$ on itself.
We mention here that in the case where $\K=\R$, $d_\fH$  is also invariant under the right action of $\fH_\K=\R^{n-1}$ to itself. The same holds for the case $\K\neq\R$, $n=1$.
\item Rotations ($n>1$): these come from the action of ${\rm F}(n-1)$ on $\K^{n-1}$. That is, given a $U\in{\rm F}(n-1)$ we define
$$
S_U(\zeta,v)=(U\cdot\zeta,v).
$$
\item Only in the case where $\K=\mH$ we have the action of ${\rm F}(1)={\rm Sp}(1)$ given by
$$
(\zeta_1,\dots,\zeta_{n-1},v)\mapsto (\mu\zeta_1\mu^{-1},\dots,\mu\zeta_{n-1}\mu^{-1},\mu v\mu^{-1}),\quad \mu\in{\rm Sp}(1);
$$
observe that in all other cases this action is vacuous.
\end{enumerate}These actions form the group ${\rm Isom}(\fH_\K,d_\fH)$ of $d_\fH-${\it isometries}; this acts transitively on $\fH_\K$. The stabiliser of $0$ consists of transformations of the form (2) (resp. of the form (2) and (3)) if $\K\neq\mH$ (resp. if $\K=\mH$). All the above transformations are extended naturally (and uniquely) on the boundary $\partial\bH_\K^n$, by requiring the extended transformations to map $\infty$ to itself.

We also consider two other kinds of transformations of $\partial\bH_\K^n$.
\begin{enumerate}
\item[{(4)}] Dilations: if $\delta\in\R_*^+$ we define
$$
D_\delta(\zeta,v)=(\delta\zeta,\delta^2v),\quad D_\delta(\infty)=\infty.
$$
It is easy to see that for every $(\zeta,v),(\zeta',v')\in\partial\bH_\K^n$ we have
$$
d_\fH\left(D_\delta(\zeta,v),D_\delta(\zeta',v')\right)=\delta\;d_\fH\left((\zeta,v),(\zeta',v')\right)
$$
and thus the metric $d_\fH$ is  scaled up to multiplicative constants by the action of  dilations. We mention here that together with $d_\fH-$isometries, dilations form the $d_\fH-${\it similarity group} ${\rm Sim}(\fH_\K,d_\fH)$.
\item[{(5)}] Inversion $R$ is given by
$$
R(\zeta,v)=\left(\zeta(-\|\zeta\|^2+v)^{-1}\;,\;\overline{v}\left|-\|\zeta\|^2+v\right|^{-2}\right),\;\;\text{if}\;(\zeta,v)\neq o,\infty,\;\quad R(o)=\infty,\;R(\infty)=o.
$$
Inversion $R$ is an involution of $\partial\bH^n_\K$. Moreover, for all $p=(\zeta,v),p'=(\zeta',v')\in\fH_\K\setminus\{o\}$ we have
$$
d_\fH(R(p),o)=\frac{1}{d_\fH(p,o)},\quad d_\fH(R(p),R(p'))=\frac{d_\fH(p,p')}{d_\fH(p,o)\;d_\fH(o,p')}.
$$
\end{enumerate}

\medskip

The group generated from similarities and inversion is isomorphic to $G_\K$ ; each transformation of $G_\K$  can be written as a composition of transformations of the form (1)--(5). The stabiliser of $0$ and $\infty$ is the subgroup comprising the (extended) transformations of the form (2),(3) and (4). Given two distinct points on the boundary, we can find an element of $G_\K$  mapping those points to $0$ and $\infty$ respectively; in particular $G_\K$  acts doubly transitively on the boundary. In the exceptional case where $\K=\R$, the action of $G_\R$  is triply transitive; this follows from the fact that we can map three distinct points of the boundary to the points $0$, $\infty$ and $(1,0,\dots,0)$ respectively.

\subsubsection{$\R-$circles}\label{sec:K-Rcircles}

An $\R-${\it circle} is the intersection of a totally real plane with the boundary $\partial\bH^n_\K$. The {\it standard} $\R-$circle (passing through $0$ and $\infty$) is the set 
$$
R_\R=\left\{(x,0,\dots,0,0)\in\fH_\K\;|\;x\in\R\right\}.
$$
Any other $\R-$circle is the image of $R_\R$ via an element of $G_\K$. Given two distinct points on the boundary, there is a (unique) $\R-$circle passing through these points. This does not hold in general for the case of three points unless $\K=\R$. We have already noted in the introduction that $\R-$circles are Circles (that is Euclidean circles and straight lines) only in the case where $\K=\R$. In all other cases the picture is quite different, see for instance \cite{Gol} or \cite{P} and \cite{KP}. Note that the above definition does not cover the case of $\R-$circles in the boundary of $\bH^1_\C$. In this case the boundary itself, i.e. the circle $S^1$ will be considered as the unique $\R-$circle.

\section{$\K-$Cross--Ratios ($\K=\R,\C,\mH$)}\label{sec:cr}

In this section we define the $\K-$cross--ratio of four distinct points in the boundary of $\bH^n_\K$ and we study some of their properties in Section \ref{sec:cr-prop}. In Section \ref{sec:cr-fund} we prove two fundamental relations in Proposition \ref{prop-norm}. In particular, Inequality \ref{eq-variety2} is the tool for the proof of our main result. 

\subsection{$\K-$Cross--Ratios and their properties}\label{sec:cr-prop}

Given a  quadruple of distinct points $\fp=(p_1,p_2,p_3,$ $p_4)$ in $\partial\bH_\K^n$,  their  cross--ratio is defined by
$$
\X(p_1,p_2,p_3,p_4)=\langle{\bf p}_4,{\bf p}_2\rangle\langle{\bf p}_4,{\bf p}_1\rangle^{-1}\langle{\bf p}_3,{\bf p}_1\rangle\langle{\bf p}_3,{\bf p}_2\rangle^{-1}
$$
where ${\bf p}_i,$ are lifts of $p_i,\;i=1,\dots, 4$ in $\K^{n,1}$. The order of multiplication plays a role only in the case where $\K=\mH$. It is clear that $[p_1,p_2,p_3,p_4]=\X(p_1,p_2,p_3,p_4)$ is invariant under the action of $G_\K$. An important observation is that the cross--ratio is independent of the choice of lifts in the cases when $\K=\R$ or $\C$ but {\it not} when $\K=\mH$.  Indeed, if ${\bf q}_i={\bf p}_i\lambda_i,\;i=1,\dots,4$ where $\lambda_i\in\K$ then 
$$
\langle{\bf q}_4,{\bf q}_2\rangle\langle{\bf q}_4,{\bf q}_1\rangle^{-1}\langle{\bf q}_3,{\bf q}_1\rangle\langle{\bf q}_3,{\bf q}_2\rangle^{-1}={\overline \lambda_2}\langle{\bf p}_4,{\bf p}_2\rangle\langle{\bf p}_4,{\bf p}_1\rangle^{-1}\langle{\bf p}_3,{\bf p}_1\rangle\langle{\bf p}_3,{\bf p}_2\rangle^{-1}{\overline \lambda_2}^{-1}.
$$
In case where $\K=\mH$ there is no cancellation of the $\lambda_i$'s; therefore in general, the cross ratios are only defined up to similarity. In other words the  invariant
quantities obtained by the cross ratio $\X$ are $|\X|$ and $\Re(\X)$.

The square root of the absolute value of the $\K-$cross--ratio is
\begin{equation*}
 |\X(p_1,p_2,p_3,p_4)|^{1/2}=\frac{d_\fH(p_4,p_2)\cdot d_\fH(p_3,p_1)}{d_\fH(p_4,p_1)\cdot d_\fH(p_3,p_2)},
\end{equation*}
where here $d_\fH$ denotes the extended metric of $\fH_\K$.
\medskip

The proof of the following proposition is by direct computations. 

\medskip

\begin{prop}\label{prop-symmetries}
The following symmetric relations hold:
\begin{eqnarray*}
 &&\label{sym1}
|[p_1,p_2,p_3,p_4]|=|[p_2,p_1,p_4,p_3]|=|[p_3,p_4,p_1,p_2]|=|[p_4,p_3,p_2,p_1]|,\\
&&\label{sym2}
\Re([p_1,p_2,p_3,p_4])=\Re([p_2,p_1,p_4,p_3])=\Re([p_3,p_4,p_1,p_2])=\Re([p_4,p_3,p_2,p_1]).
\end{eqnarray*}
\end{prop}

\medskip

Proposition \ref{prop-symmetries} tells us that for a given quadruple of distinct points in $\partial{\bf H}^n_\K$, the moduli and the real parts of all 24 quaternionic cross--ratios are real analytic functions of the moduli and the real parts respectively of the following three cross--ratios:
\begin{eqnarray}
&&\label{X1}
 \X_1=[p_1,p_2,p_3,p_4]=\langle{\bf p}_4,{\bf p}_2\rangle\langle{\bf p}_4,{\bf p}_1\rangle^{-1}
\langle{\bf p}_3,{\bf p}_1\rangle\langle{\bf p}_3,{\bf p}_2\rangle^{-1},\\
&&\label{X2}
\X_2=[p_1,p_3,p_2,p_4]=\langle{\bf p}_4,{\bf p}_3\rangle\langle{\bf p}_4,{\bf p}_1\rangle^{-1}
\langle{\bf p}_2,{\bf p}_1\rangle\langle{\bf p}_2,{\bf p}_3\rangle^{-1},\\
&&\label{X3}
\X_3=[p_2,p_3,p_1,p_4]=\langle{\bf p}_4,{\bf p}_3\rangle\langle{\bf p}_4,{\bf p}_2\rangle^{-1}
\langle{\bf p}_1,{\bf p}_2\rangle\langle{\bf p}_1,{\bf p}_3\rangle^{-1}.
\end{eqnarray}
In fact, we have

\medskip

\begin{prop}
Let $\X_i$, $i=1,2,3$ be as in \ref{X1}, \ref{X2} and \ref{X3} respectively. 
Then,
\begin{eqnarray*}
 &&\label{sym3}
[p_1,p_2,p_4,p_3]=\X_1^{-1},\\
&&\label{sym4}
[p_1,p_3,p_4,p_2]=\X_2^{-1},\\
&&\label{sym5}
|[p_1,p_4,p_3,p_2]|=1/|\X_3|,\quad \Re([p_1,p_4,p_3,p_2])=\Re(\X_3^{-1}),\\
&&\label{sym6}
|[p_1,p_4,p_2,p_3]|=|\X_3|,\quad \Re([p_1,p_4,p_2,p_3])=\Re(\X_3).
\end{eqnarray*}
\end{prop}

\medskip

In the cases where $\K=\R$ or $\C$, the 24 cross--ratios corresponding to a quadruple $\fp=(p_1,p_2,p_3,$ $p_4)$ are themselves real analytic functions of $\X_i$, $i=1,2,3$. Moreover, in the case where $\K=\R$, the cross--ratio of any quadruple is a real number (different from 0 and 1) and $\X_3$ is just $\X_2/\X_1$. Therefore in this case the cross--ratios of a given quadruple depend only in the (real) cross--ratios $\X_1$ and $\X_2$. A trivial case appears in the case where $n=1$ and $\K=\C$. Then we have the relations
$$
\X_3=-\X_2\X_1^{-1}\quad\text{and}\quad \X_1+\X_2=1.
$$
The same relations hold in the case where $n=1$ and $\K=\mH$ but they are dependent from the choice of lifts. For quaternionic cross--ratios in the one dimensional case, see \cite{BG} and \cite{GL}.

\subsection{Two fundamental relations for $\K-$cross--ratios}\label{sec:cr-fund}

Two relations concerning cross--ratios which hold in any case and are independent of the choice of lifts are given in the next proposition. These relations are well known as equalities defining Falbel's cross--ratio variety in the case where $n=2$ and $\K=\C$, see for instance \cite{F}, \cite{FP} and \cite{PP}. For a somewhat different treatment of this case, see also \cite{CG}. In our setting, these equalities are following from the next general result.

\medskip

\begin{prop}\label{prop-norm}
 Let $\fp=(p_1,p_2,p_3,p_4)$ be a quadruple of four distinct points in $\partial{\bf H}^n_\K$. Let also $\X_1,\X_2$ and $\X_3$ 
be defined by \ref{X1}, \ref{X2} and \ref{X3}. Then
\begin{eqnarray}
 &&\label{eq-variety1}
|\X_2|=|\X_1|\,|\X_3|,\\
&&\label{eq-variety2}
2|\X_1|^2\Re(\X_3)\ge|\X_1|^2+|\X_2|^2-2\Re(\X_1)-2\Re(\X_2)+1.
\end{eqnarray}
\end{prop}
\begin{proof}
Equation \ref{eq-variety1} is evident from the definitions of $\X_i$, $i=1,2,3$. To show the validity 
of Inequality \ref{eq-variety2} we observe that due to the double--transitive 
action of ${\rm PF}(n,1)$ on the boundary, it is always possible to normalise the points so that
the quadruple is
$$
p_1=\infty, \quad p_2=(\zeta_2,v_2),\quad p_3=(\zeta_3,v_3),\quad p_4=o,
$$
where  $\zeta_i=(\zeta_i^1,\dots,\zeta_i^{n-1})$, $i=2,3$.
We consider standard lifts
\begin{equation*}
{\bf p}_1=\left[\begin{matrix} 1 \\ 0 \\ 0 \end{matrix}\right],\quad
{\bf p}_2=\left[\begin{matrix} -\|\zeta_2\|^2+v_2 \\ \sqrt{2}\zeta_2 \\ 1 \end{matrix}\right],\quad
{\bf p}_3=\left[\begin{matrix} -\|\zeta_3\|^2+v_3 \\ \sqrt{2}\zeta_3\\ 1 \end{matrix}\right],\quad
{\bf p}_4=\left[\begin{matrix} 0 \\ 0 \\ 1 \end{matrix}\right]
\end{equation*}
and we calculate
\begin{eqnarray*}
&&
\langle\bp_1,\bp_2\rangle=\langle\bp_1,\bp_3\rangle=\langle\bp_1,\bp_4\rangle=1,\\
&&
\langle\bp_2,\bp_3\rangle=-\|\zeta_2\|^2-\|\zeta_3\|^2+2\langle\langle \zeta_2,\zeta_3\rangle\rangle+v_2+\overline{v_3},\quad 
\langle\bp_2,\bp_4\rangle=-\|\zeta_2\|^2+v_2,\\
&&
\langle\bp_3,\bp_4\rangle=-\|\zeta_3\|^2+v_3.
\end{eqnarray*}
Therefore
\begin{eqnarray*}
&&
\X_1=\left(-\|\zeta_2\|^2+\overline{v_2}\right)\left(-\|\zeta_2\|^2-\|\zeta_3\|^2+2\langle\langle\zeta_3,\zeta_2\rangle\rangle+\overline{v_2}+v_3\right)^{-1},\\
&&
\X_2=\left(-\|\zeta_3\|^2+\overline{v_3}\right)\left(-\|\zeta_2\|^2-\|\zeta_3\|^2+2\langle\langle\zeta_2,\zeta_3\rangle\rangle+\overline{v_3}+v_2\right)^{-1},\\
&&
\X_3=\left(-\|\zeta_3\|^2+\overline{v_3}\right)\left(-\|\zeta_2\|^2+\overline{v_2}\right)^{-1}.
\end{eqnarray*}
We set $$p=-\|\zeta_2\|^2-\|\zeta_3\|^2+2\langle\langle \zeta_3,\zeta_2\rangle\rangle+\overline{v_2}+v_3.$$ Then
\begin{eqnarray*}
2|p|^2|\X_1|^2\Re(\X_3)&=&2\left|-\|\zeta_2\|^2+\overline{v_2}\right|^2\;\Re\left(\frac{(-\|\zeta_3\|^2+\overline{v_3})(-\|\zeta_2\|^2+v_2)}{\left|-\|\zeta_2\|^2+\overline{v_2}\right|^2}\right)\\
&=&2\left(\|\zeta_3\|^2\|\zeta_2\|^2-\Re(v_2v_3)\right)\\
&\ge&4\left|\langle\langle\zeta_3,\zeta_2\rangle\rangle\right|^2-2\left(\|\zeta_2\|^2\|\zeta_3\|^2+\Re(v_2v_3)\right)\\
&=&\left|(-\|\zeta_2\|^2+\overline{v_2})+(-\|\zeta_3\|^2+v_3)-p\right|^2-2\Re\left((-\|\zeta_2\|^2+v_2)(-\|\zeta_3\|^2+\overline{v_3})\right)\\
&=&\left|-\|\zeta_2\|^2+\overline{v_2}\right|^2+\left|-\|\zeta_3\|^2+\overline{v_3}\right|^2\\
&&-2\Re\left((-\|\zeta_2\|^2+\overline{v_2})\overline{p}\right)-2\Re\left((-\|\zeta_3\|^2+\overline{v_3})p\right)+|p|^2\\
&=&|p|^2\left(|\X_1|^2+|\X_2|^2-2\Re(\X_1)-2\Re(\X_2)+1\right).
\end{eqnarray*}
Here, the inequality in the third line follows from Cauchy--Schwarz inequality (which holds in every inner product space). 
\end{proof}

\medskip

Although Equation \ref{eq-variety1} of Proposition \ref{prop-norm} is obvious from the symmetry conditions, Inequality \ref{eq-variety2} is neither trivial nor obvious at all. It is natural to ask when it holds as an equality. We observe that 
for $n>2$, there are quadruples so that Inequality \ref{eq-variety2} is strict. Take for instance $p_i$, $i=1,\dots,4$ to be
$$
p_1=\infty,\quad p_2=(1,0,\dots,0,0),\quad p_3=(0,0,\dots,-1,0),\quad p_4=o. 
$$
Then $\X_1=\X_2=1/2$, $\X_3=1$ and
$$
\frac{1}{2}=2|\X_1|^2\Re(\X_3)>|\X_1|^2+|\X_2|^2-2\Re(\X_1)-2\Re(\X_2)+1=-\frac{1}{2}.
$$
The cases where Inequality \ref{eq-variety2} holds as an equality are treated in the next proposition.

\medskip

\begin{prop}\label{prop-norm-eq}
With the hypotheses of Proposition \ref{prop-norm}, Inequality \ref{eq-variety2} holds as an equality if and only if one of the following cases occur:
\begin{enumerate}
\item[{(i)}] $n=1$, $\K\neq \R$.
\item[{(ii)}] $n=2$.
\item[{(iii)}] $p_1,p_2,p_4$ or $p_1,p_3,p_4$  lie in the same $\K-$line.
\item[{(iv)}] All points of $\fp$ lie in the same $\K-$line.
\item[{(v)}] The points $p'_i=\Pi_{\K^{n-1}}(p_i)$, $i=2,3$ define the same point in the projective space $\K P^{n-2}$. Here $\Pi_{\K^{n-1}}$ is the natural projection from $\partial\bH^n_\K$ to $\K^{n-1}\cup\{\;\infty\}$ given by
$$
\Pi_{\K^{n-1}}(\zeta,v)=\zeta\;\;\text{if}\;\;(\zeta,v)\in\fH_\K,\quad \Pi_{\K^{n-1}}(\infty)=\infty.
$$
\end{enumerate}
\end{prop}

\begin{proof}
 From the proof of Inequality \ref{eq-variety2} we see that for $n=2$ (and vacuously for $n=1$, $\K\neq\R$) this inequality is in fact an equality; these are cases  (i) and (ii). Under our normalisation, the Cauchy--Schwarz inequality 
$$
\left|\langle\langle\zeta_2,\zeta_3\rangle\rangle\right|\le\|\zeta_2\|\;\|\zeta_3\|
$$
holds as an equality if and only if one or both of the $\zeta_i$'s are zero, or there exists a $\lambda\in\K$ such that $\zeta_2=\lambda\zeta_3$. The first case gives (iii) and (iv). The latter case gives (v) and the proof is complete.
\end{proof}

\medskip

We further remark that in the case where $\K=\R$, $n\neq 1$, the cross--ratio of a quadruple of four points and thus all $\X_i$, $i=1,2,3$ are positive numbers and therefore are identical to their absolute values. In fact, under our normalisation we have
$$
\X_1=\frac{\|\zeta_2\|^2}{\|\zeta_2-\zeta_3\|^2},\quad \X_2=\frac{\|\zeta_3\|^2}{\|\zeta_2-\zeta_3\|^2}
$$
and we have already mentioned that $\X_3=\X_2/\X_1$. Also, in case (iii) we have $\X_i\in\R$ and
$$
\X_1+\X_2=1,\quad \X_3=-\X_2/\X_1
$$
and obviously this case (as well as case (iii)) never appears when $\K=\R$. Finally, the statement of (v) may be also read as

\medskip

\begin{enumerate}
\item[{(v*)}] The points $p_2'$ and $p_3'$ are in the same orbit of the stabiliser of $p_1$ and $p_4$.
\end{enumerate}

\medskip

For $\K=\R$ this is equivalent to $p_2$ and $p_3$ are in the same orbit of the stabiliser of $p_1$ and $p_4$ as well and thus all points of $\fp$ lie in an $\R-$circle.

\section{Ptolemaean Inequality and Ptolemaeus' Theorem I}\label{sec:pt1}

In this section we prove  the  $\partial\bH_\K^n$ ($\K=\R,\C,\mH$)  version (and subsequently the $\fH_\K$ version) of the Ptolemaean inequality and  of the Theorem of Ptolemaeus respectively; these are derived almost immediately from the properties of cross--ratios and especially from Inequality \ref{eq-variety2}. We start by proving the Ptolemaean inequality in $\partial\bH^n_\K$.

\medskip

\begin{thm}\label{thm:ptolemaean-ineq}{\bf (Ptolemaean inequality  in $\partial\bH^n_\K$)}
Let $\fp=(p_1,p_2,$ $p_3,p_4)$ a quadruple of distinct points in $\partial \bH_\K^n$,  $\K=\R,\C,\mH$, and consider the cross--ratios $\X_i=\X_i(\fp)$, $i=1,2,$ defined by \ref{X1} and \ref{X2} respectively. Then  the following inequalities hold:
 \begin{equation}\label{eq:ptolemaean-inequality-cr}
  |\X_1|^{1/2}+|\X_2|^{1/2}\ge 1,\quad \text{and}\quad -1\le|\X_1|^{1/2}-|\X_2|^{1/2}\le 1.
 \end{equation}
 In the case where $\K=\R$ the absolute values may be omitted.
\end{thm}
\begin{proof}
From Inequality \ref{eq-variety2} and using the obvious inequality $\Re(\X_i)\le|\X_i|$, $i=1,2,3$,  we have
\begin{eqnarray*}
0&\ge& |\X_1|^2+|\X_2|^2-2\Re(\X_1)-2\Re(\X_2)+1-2|\X_1|^2\Re(\X_3)\\
&\ge & |\X_1|^2+|\X_2|^2-2|\X_1|-2|\X_2|+1-2|\X_1|^2|\X_3|\\
\text{from\;Eq. \ref{eq-variety1}}&=&|\X_1|^2+|\X_2|^2-2|\X_1|-2|\X_2|+1-2|\X_1||\X_2|\\
&=&(|\X_1|+|\X_2|-1)^2-4|\X_1||\X_2|\\
&=&(|\X_1|+|\X_2|-2|\X_1|^{1/2}|\X_2|^{1/2}-1)\cdot (|\X_1|+|\X_2|+2|\X_1|^{1/2}|\X_2|^{1/2}-1)\\
&=&\left((|\X_1|^{1/2}-|\X_2|^{1/2})^2-1\right)\cdot\left((|\X_1|^{1/2}+|\X_2|^{1/2})^2-1\right),
\end{eqnarray*}
and observe that in the case $\K=\R$ the second inequality holds as an equality. Therefore,
$$
(|\X_1|^{1/2}-|\X_2|^{1/2}-1)\cdot(|\X_1|^{1/2}-|\X_2|^{1/2}+1)\cdot(|\X_1|^{1/2}+|\X_2|^{1/2}-1)\cdot(|\X_1|^{1/2}+|\X_2|^{1/2}+1)\le 0
$$
and since $|\X_1|^{1/2}+|\X_2|^{1/2}+1>0$, this reduces to
$$
(|\X_1|^{1/2}-|\X_2|^{1/2}-1)\cdot(|\X_1|^{1/2}-|\X_2|^{1/2}+1)\cdot(|\X_1|^{1/2}+|\X_2|^{1/2}-1)\le 0.
$$
Suppose that $\left(|\X_1|^{1/2}-|\X_2|^{1/2}\right)^2> 1$; then 
$$
1<\left(|\X_1|^{1/2}-|\X_2|^{1/2}\right)^2<\left(|\X_1|^{1/2}+|\X_2|^{1/2}\right)^2
$$
which is a contradiction.
Therefore, $\left(|\X_1|^{1/2}-|\X_2|^{1/2}\right)^2\le 1$ and $\left(|\X_1|^{1/2}+|\X_2|^{1/2}\right)^2\ge 1$ 
which proves the Ptolemaean inequality.
\end{proof}

\medskip

Ptolemaus' Theorem in $\partial\bH^n_\K$ is in order next.

\medskip

\begin{thm}{(\bf Ptolemaeus' Theorem in $\partial\bH^n_\K$)}\label{thm:ptolemaean-eq}
 Each inequality  \ref{eq:ptolemaean-inequality-cr} holds if and only if all four points of $\fp$ lie in an $\R-$circle. Then, $\X_i>0$, $i=1,2$ and 
\begin{enumerate}
 \item $\X_1^{1/2}-\X_2^{1/2}=1$ if $p_1$ and $p_3$ separate $p_2$ and $p_4$;
 \item $\X_2^{1/2}-\X_1^{1/2}=1$ if  $p_1$ and $p_2$ separate $p_3$ and $p_4$;
 \item $\X_1^{1/2}+\X_2^{1/2}=1$ if  $p_1$ and $p_4$ separate $p_2$ and $p_3$.
\end{enumerate}
\end{thm}
\begin{proof}
Suppose first that one of the inequalities holds as an equality. We are going to prove first that all $\X_i$, $i=1,2,3$ are positive. This is already known for the case where $\K=\R$ but the following arguments do not affect this fact. We have,
$$
(|\X_1|^{1/2}-|\X_2|^{1/2}-1)\cdot(|\X_1|^{1/2}-|\X_2|^{1/2}+1)\cdot(|\X_1|^{1/2}+|\X_2|^{1/2}-1)\cdot(|\X_1|^{1/2}+|\X_2|^{1/2}+1)= 0,
$$
which is equivalent to
$$
(|\X_1|-|\X_2|)^2= 2(|\X_1|+|\X_2|)-1.
$$
Since $
(|\X_1|-|\X_2|)^2\le 2\Re(\X_1+\X_2)-1,
$
we have $|\X_1|+|\X_2|\le \Re(\X_1+\X_2)$ and therefore
$$
0\ge\Re(\X_2)-|\X_2|\ge |\X_1|-\Re(\X_1)\ge 0.
$$
Thus $\X_1,\X_2$ are positive. Now from Inequality \ref{eq-variety2} it follows
\begin{eqnarray*}
 2\X_1^2\Re(\X_3)&\ge&\X_1^2+\X_2^2-2\X_1-2\X_2+1\\
&=&\X_1^2+\X_2^2-(\X_1-\X_2)^2=2\X_1\X_2\\
\text{using \;Eq.}\; \ref{eq-variety1}\quad &=&2\X_1^2|\X_3|,
\end{eqnarray*}
and thus $\X_3>0$. Summing up, we have that $\X_i>0$ for each $i=1,2,3$ and moreover, since the above inequality holds as an equality we have 
\begin{eqnarray*}
 0&=&\X_1^2+\X_2^2+1-2\X_1-2\X_2-2\X_1\X_2\\
&=&\left(\X_1^{1/2}+\X_2^{1/2}+1\right)\left(\X_1^{1/2}+\X_2^{1/2}-1\right)\left(\X_1^{1/2}-\X_2^{1/2}+1\right)\left(\X_1^{1/2}-\X_2^{1/2}-1\right)
\end{eqnarray*}
and thus at least one of the three equalities in the statement of the theorem holds true. 


With no loss of generality we may suppose that the equation in question is
$$
\X_1^{1/2}-\X_2^{1/2}=1
$$
and we may also assume as in the proof of Proposition \ref{prop-norm} that $\fp$ is the quadruple
$$
p_1=\infty, \quad p_2=(\zeta_2,v_2),\quad p_3=(\zeta_3,v_3),\quad p_4=o,
$$
where  $\zeta_i=(\zeta_1^1,\dots,\zeta_i^{n-1})$, $i=2,3$. Then,
\begin{eqnarray*}
&&
\X_1=\left(-\|\zeta_2\|^2+\overline{v_2}\right)\left(-\|\zeta_2\|^2-\|\zeta_3\|^2+2\langle\langle\zeta_3,\zeta_2\rangle\rangle+\overline{v_2}+v_3\right)^{-1},\\
&&
\X_2=\left(-\|\zeta_3\|^2+\overline{v_3}\right)\left(-\|\zeta_2\|^2-\|\zeta_3\|^2+2\langle\langle\zeta_2,\zeta_3\rangle\rangle+\overline{v_3}+v_2\right)^{-1},\\
&&
\X_3=\left(-\|\zeta_3\|^2+\overline{v_3}\right)\left(-\|\zeta_2\|^2+\overline{v_2}\right)^{-1}.
\end{eqnarray*}
From $\X_3=a_3>0$ we have $-\|\zeta_3\|^2+\overline{v_3}=-a_3\|\zeta_2\|^2+a_3\overline{v_2}$, therefore
$$
\|\zeta_3\|=a_3^{1/2}\|\zeta_2\|,\quad v_3=a_3\;v_2.
$$
Also, from $\X_i=a_i>0$, $i=1,2$ we obtain the relations
\begin{eqnarray}
&&\label{eq:rel1}
\left[1-a_1(a_3+1)\right]\|\zeta_2\|^2+2a_1\langle\langle\zeta_3,\zeta_2\rangle\rangle+(a_1-1)\overline{v_2}+a_1a_3v_2=0,\\
&&\label{eq:rel2}
\left[a_3-a_2(a_3+1)\right]\|\zeta_2\|^2+2a_2\langle\langle\zeta_2,\zeta_3\rangle\rangle+a_3(a_2-1)\overline{v_2}+a_2v_2=0.
\end{eqnarray}
Since $a_2=a_1a_3$ we write the Equation \ref{eq:rel2} as
$$
\left[\frac{a_2}{a_1}-a_1a_3(a_3+1)\right]\|\zeta_2\|^2+2a_1a_3\langle\langle\zeta_2,\zeta_3\rangle\rangle+a_3(a_1a_3-1)\overline{v_2}+a_1a_3v_2=0
$$
and by taking out $a_3$ as a common factor and then taking conjugates, we write this as
$$
\left[1-a_1(a_3+1)\right]\|\zeta_2\|^2+2a_1\langle\langle\zeta_3,\zeta_2\rangle\rangle+(a_1a_3-1)v_2+a_1\overline{v_2}=0.
$$
Subtracting from the Equation \ref{eq:rel1} we obtain $v_2-\overline{v_2}=0$, therefore $v_2=0$ and also $v_3=0$. It is now clear that $\langle\langle\zeta_2,\zeta_3\rangle\rangle\in\R$, thus
$$
\X_1=\frac{\|\zeta_2\|^2}{\|\zeta_2-\zeta_3\|^2},\quad \X_2=\frac{\|\zeta_3\|^2}{\|\zeta_2-\zeta_3\|^2}.
$$
From $\X_1^{1/2}-\X_2^{1/2}=1$ we have $\|\zeta_2\|-\|\zeta_3\|=\|\zeta_2-\zeta_3\|$, therefore there exists a positive $\lambda>1$ such that $\zeta_2=\lambda\;\zeta_3$ from where it follows that the quadruple $\fp$ lies in the same $\R-$circle. Moreover, $p_1$ and $p_3$ separate $p_2$ and $p_4$.

\medskip

Conversely, if all points lie on an $\R$-circle, we may suppose that $p_1$ and $p_3$ separate $p_2$ and $p_4$ and conjugate so that
$$
p_1=\infty,\quad p_2=(\lambda,0,\dots 0,0),\quad p_3=(1,0,\dots,0,0),\quad p_4=o,
$$ 
where 
$\lambda\in\R$, $\lambda>1$. 
Then $\X_1=\lambda^2/(\lambda-1)^2$ and $\X_2=1/(\lambda-1)^2$ from where we obtain
$$
\X_1^{1/2}-\X_2^{1/2}=1.
$$
Rearranging the points, we obtain in a similar manner the other two statements of the Theorem.
\end{proof}

\medskip

Let now $\fp=(p_1,p_2,p_3,p_4)$ a quadruple of distinct points in the  group $\fH_\K$. Inequalities \ref{eq:ptolemaean-inequality-cr} of Theorem \ref{thm:ptolemaean-ineq} can be written as
\begin{eqnarray}\label{eq:ineq-heis1}
&&
d_\fH(p_2,p_3)\cdot d_\fH(p_1,p_4)\le d_\fH(p_2,p_4)\cdot d_\fH(p_1,p_3)+d_\fH(p_1,p_2)\cdot d_\fH(p_3,p_4),\\
&&\label{eq:ineq-heis2}
d_\fH(p_1,p_3)\cdot d_\fH(p_2,p_4)\le d_\fH(p_1,p_2)\cdot d_\fH(p_3,p_4)+d_\fH(p_2,p_3)\cdot d_\fH(p_1,p_4),\\
&&\label{eq:ineq-heis3}
d_\fH(p_1,p_2)\cdot d_\fH(p_3,p_4)\le d_\fH(p_1,p_3)\cdot d_\fH(p_2,p_4)+d_\fH(p_2,p_3)\cdot d_\fH(p_1,p_4).
\end{eqnarray}

\medskip

As a corollary of the above discussion we derive the following.

\medskip

\begin{thm}\label{thm:ptolemaean-heis} {\bf (Ptolemaean inequality and  Ptolemaeus' Theorem in the group $\fH_\K$)}
  The  metric $d_\fH$ in the group $\fH_\K$, $\K=\R,\C,\mH$, satisfies the Ptolemaean inequality: for each quadruple of points $\fp=(p_1,p_2,p_3,p_4)$, inequalities \ref{eq:ineq-heis1}, \ref{eq:ineq-heis2} and \ref{eq:ineq-heis3} hold. Moreover, each of these inequalities hold as an equality if and only if all points lie in an $\R-$circle. Explicitly,
\begin{enumerate}
\item $d_\fH(p_2,p_3)\cdot d_\fH(p_1,p_4)=d_\fH(p_2,p_4)\cdot d_\fH(p_1,p_3)+d_\fH(p_1,p_2)\cdot d_\fH(p_3,p_4)$ if and only if all points lie in an $\R-$circle and $p_1$ and $p_4$ separate $p_2$ and $p_3$;
\item $d_\fH(p_1,p_3)\cdot d_\fH(p_2,p_4)=d_\fH(p_1,p_2)\cdot d_\fH(p_3,p_4)+d_\fH(p_2,p_3)\cdot d_\fH(p_1,p_4)$ if and only if all points lie in an $\R-$circle and $p_1$ and $p_3$ separate $p_2$ and $p_4$;
\item $d_\fH(p_1,p_2)\cdot d_\fH(p_3,p_4)=d_\fH(p_1,p_3)\cdot d_\fH(p_2,p_4)+d_\fH(p_2,p_3)\cdot d_\fH(p_1,p_4)$ if and only if all points lie in an $\R-$circle and $p_1$ and $p_2$ separate $p_3$ and $p_4$.
\end{enumerate}
\end{thm}

\section{Preliminaries for the Case $\K=\mO$}\label{sec:prelO}

The material in this section is from \cite{MP} and the papers referenced therein; the treatment of octonionic hyperbolic plane follows closely that given in the paper of Allcock, \cite{A}. For reasons of consistency  we have kept the notation of \cite{MP} almost intact. In Section \ref{sec:oct} we describe briefly the set of octonions; the construction of octonionic hyperbolic plane  $\bH^2_\mO$ follows in Section \ref{sec:oct-hyp}. The boundary $\partial\bH^2_\mO$, the analogue of the octonionic Heisenberg group and the metric $d_\fH$, as well as the action of the group of isometries of $\bH^2_\mO$ in the boundary are described in Section \ref{sec:bound-O}. Finally, $\R-$circles on the boundary are described in Section \ref{sec:R-sirc-O}.

\subsection{Octonions}\label{sec:oct}
The set of octonions is the 8--dimensional real vector space with basis $\be_0=1$ and $\be_i$, $i=1,\dots,7$ together with a non associative multiplication defined in the basis vectors by the following rules.
\begin{enumerate}
\item[{(i)}]
$
\be_0\be_i=\be_i\be_0=\be_i$, $\be_i^2=-1$, $i=1,\dots,7;
$ 
\item[{(ii)}]
$
\be_i\be_j = -\delta_{ij}\be_0 + \varepsilon _{ijk} \be_k, 
$
where $\delta_{ij}$ is Kronecker's delta tensor and $\varepsilon _{ijk}$ is a completely antisymmetric tensor with value +1 when $ijk = 124$, $137$, $156$, $235$, $267$, $346$, $457$.
\end{enumerate}
Multiplication is extended everywhere in $\mO$ by linearity. 

We write an octonion $z$ as
$
z=z_0+\sum_{i=1}^7z_i\be_i 
$.
Its  conjugate is defined to be
$
\overline{z}=z_0-\sum_{i=1}^7z_i\be_i 
$ and for any two octonions $z$ and $w$ we have $\overline{zw}=\overline{w}\;\overline{z}$. 
The real part of an octonion $z$ is $\Re(z)=(z+\overline{z})/2=z_0$ whether its imaginary part is $\Im(z)=(z-\overline{z})/2=\sum_{i=1}^7z_i\be_i$. 
The modulus $|z|$ of an octonion is the non--negative real number defined by $|z|^2=z\overline{z}=\overline{z}z=\sum_{i=0}^7z_i^2$. For any two octonions $z$ and $w$ we have $|zw|=|z|\;|w|$ and $|z|=0$ if and only if $z=0$. The inverse of an non zero octonion $z$ is the octonion $z^{-1}=\overline{z}/|z|^2$. Clearly, $z\;z^{-1}=z^{-1}\;z=1$. A unit octonion is an octonion $\mu$ with $|\mu|=1$. The inverse of a unit octonion $\mu$ is its conjugate $\overline{\mu}$. When a unit octonion $\mu$ is purely imaginary, its inverse is its opposite $-\mu$.

\medskip

The following Proposition is found in Propositions 3.1 and 3.2 of \cite{MP}.

\medskip

\begin{prop}
\begin{enumerate}
\item For any octonions $x$ and $y$  the subalgebra with a unit generated by 
$x$  and $y$  is  associative.    In  particular,  any  product  of  octonions  that  may  be  written 
in terms of just two octonions is associative. 
\item Suppose  that  $x$,  $y$,  $z$ are  octonions  and $\mu$  is  an  imaginary  unit octonion. Then 
\begin{eqnarray*}
&&
                           z(xy)z    =   (zx)(yz),\\ 
                                                    &&              
                            \Re((xy)z)    =     \Re(x(yz))     =     \Re((yz)x), \\                    
&&
                        (\mu x\overline{\mu})(\mu y)    =   \mu(xy),\\                                         
&&
                        (x\mu)(\overline{\mu}y\mu)    =   (xy)\mu,\\                                         
&&
                          xy + yx    =   (x\overline{\mu})(\mu y) + (y\overline{\mu})(\mu x).                          
\end{eqnarray*}
Any of the three expressions in the second relation is denoted by $\Re(xyz)$.
\end{enumerate}
\end{prop}

\subsection{Octonionic hyperbolic plane}\label{sec:oct-hyp}
Let 
$$
J=\left[\begin{array}{ccc}
0 & 0 & 1\\
0 & 1 & 0\\
1 & 0 & 0\end{array}\right],
$$
and let also ${\rm M}(3,\mO)$ be the real vector space of $3\times3$ matrices with octonionic entries. 
Let $X^*$  denote the conjugate transpose of a matrix $X$  in ${\rm M}(3,\mO)$.  Define 
$$
                          \cJ =\{ X\in {\rm M}(3,\mO) \;|\;   J X  = X^*J\}. 
$$
Then $\cJ$  is closed under the Jordan multiplication 
$$ 
X * Y   =\frac{1}{2}   (XY  + Y X), 
$$
and  $\cJ$ will thus be called  the  Jordan algebra associated to $J$.  The set of real numbers act on ${\rm M}(3,\mO)$ 
by  multiplication  of  each  entry  of  $X$. An  equivalence  relation  is defined on  $\cJ$  by
$$ 
X\sim Y\quad  \text{if and only if}\quad Y  = kX\;\;\text{for some non-zero real number}\; k. 
$$
Then we define $\mP\cJ$ to be the set of equivalence classes $[X]$.
We also define
$$
\mO^3_0
=\left\{\bv =\left(\begin{matrix}
x\\
y\\
z\\
\end{matrix}\right)\;\;|\;\; x, y, z \;\text{all lie in some associative subalgebra of}\; \mO\right\}.
$$
An equivalence relation is defined on $\mO^3_0$  by $\bv\sim\bw$ if $\bw = \bv a$ for some $a$ in an
associative subalgebra of $\mO$ containing the entries $x$, $y$, $z$ of $\bv$. Let $\mP\mO^3_0$ be the set
of equivalence classes $[\bv]$. Define a map $\pi_J : \mO^3_0\to \cJ$ by
$$
\pi_J(\bv) = \bv\bv^* J=
\left(\begin{matrix}
x\overline{z}& x\overline{y} &|x|^2\\
y\overline{z}& |y|^2 & y\overline{x}\\
|z|^2 & z\overline{y}& z\overline{x}
\end{matrix}\right).
$$
One can easily check that if $x$, $y$, $z$, $a$ all lie in an associative subalgebra of $\mO$ then
$$
\pi_J\left(\begin{matrix}
xa\\
ya\\
za
\end{matrix}\right)= |a|^2\pi_J
\left(\begin{matrix}
x\\
y\\
z\\
\end{matrix}\right).
$$
Therefore the map $\pi_J : \mP\mO^3_0\to \mP\cJ$ given by $\pi_J[\bv] =\left[\pi_J(\bv)\right]$ is well defined.
An (indefinite) norm $|\bv|_J$ is given  on $\mO^3_0$ by
$$
|\bv|_J = \bv^*J\bv ={\rm tr}\left(\pi_J(\bv)\right).
$$
The octonionic hyperbolic plane $\bH^2_\mO$ is the subset of $\mP\cJ$ comprising
$\left[\pi_J(\bv)\right]=
\pi_J[\bv]$ for $\bv\in \mO^3_0$ with $|\bv|_J < 0$. 

An invariant metric can be defined in $\bH^2_\mO$, see \cite{M} p.139. The isometry group of this metric is ${\rm Aut}(\cJ)$ which is known to be $F_{4(-20)}$, an exceptional connected 52--dimensional Lie group (see \cite{A}).

\subsection{The boundary and the variety $\fH^{15}$}\label{sec:bound-O}
The boundary of $\bH^2_\mO$ is the subset of $\mP\cJ$ comprising of $\left[\pi_J(\bv)\right]$
for $\bv\in \mO^3_0$ with $|\bv|_J < 0$. Below we present a more familiar representation of the boundary of the octonionic hyperbolic plane.
We consider the following 15--real dimensional subvariety of $\mO\times\mO$:
\begin{equation}\label{eq:H15}
\fH^{15}=\left\{(x,y)\in\mO\times\mO\;|\;2\Re(x)+|y|^2=0\right\},
\end{equation}
and we denote by $\overline{\fH^{15}}$ the set $\fH^{15}\cup\{\infty\}$. Consider the map $\psi : \overline{\fH^{15}}\to \cJ$ given by
$$
\psi(x, y) = \pi_J
\left(\begin{matrix}
x\\
y\\
1
\end{matrix}\right)=
\left(\begin{matrix}
x & x\overline{y} &|x|^2\\
y &|y|^2 & y\overline{x}\\
1& \overline{y} &\overline{x}
\end{matrix}\right),\quad
\psi(\infty) = \pi_J
\left(\begin{matrix}
1\\
0\\
0\\
\end{matrix}\right) =
\left(\begin{matrix}
0 & 0 & 1\\
0 & 0 & 0\\
0 & 0 & 0\\   
\end{matrix}\right).
$$
The map $\mP\psi: \overline{\fH^{15}}\to\partial\bH^2_\mO\subset\mP\cJ$ is a bijection between $\overline{\fH^{15}}$
and  the boundary of the octonionic hyperbolic plane which from now on will be identified to $\overline{\fH^{15}}$. 
We consider the following self transformations of $\overline{\fH^{15}}$:
\begin{enumerate}
\item left translations: for given $(t,s)\in\fH^{15}$ define
\begin{equation*}\label{eq:oct-translations}
T_{(t,s)}(x,y)=(t,s)*(x,y)=(t+x-\overline{s}y,s+y),\quad T_{(t,s)}(\infty)=\infty;
\end{equation*}
\item transformations $S_\mu$: for given unit imaginary octonion $\mu$ define
\begin{equation*}\label{eq:oct-rotations}
S_{\mu}(x,y)=(\mu x\overline{\mu}, y\overline{\mu}),\quad S_{\mu}(\infty)=\infty;
\end{equation*}
\item dilations: for given positive number $\delta$ define
\begin{equation*}\label{eq:oct-dilations}
D_{\delta}(x,y)=(\delta^4x,\delta^2y),\quad D_{\delta}(\infty)=\infty
\end{equation*}
and finally
\item inversion: for $x\neq 0$ define
\begin{equation*}\label{eq:oct-inversion}
R(x,y)=\left(\frac{\overline{x}}{|x|^2},-\frac{y\overline{x}}{|x|^2}\right),\quad R(\infty)=(0,0),\quad R(0,0)=\infty.
\end{equation*}
\end{enumerate}
We remark that in general $S_\mu\circ S_\nu\neq S_{\mu\nu}$ for $\mu,\nu$ unit imaginary octonions. The group generated by transformations $S_\mu$ is the compact group ${\rm Spin}_7(\R)$.

\medskip

Let $G_\mO$ be the group generated  by $R$, $S_\mu$ and $T_{(t,s)}$ for all unit imaginary octonions $\mu$ and all $(t,s)\in\fH^{15}$. It is proved (see \cite{A} and also \cite{MP}, minding that the first author does not contain dilations into $G_\mO$) that $G_\mO$ is isomorphic to ${\rm Aut}(\cJ)$, i.e. to $F_{4(-20)}$. We mention that the stabiliser of $o$ and $\infty$ in $G_\mO$ comprises of dilations and transformations $S_\mu$ and thus is isomorphic to $\R_+\times{\rm Spin}_7(\R)$. 

\medskip

\begin{prop}
Let $p=(x,y)$ and $q=(w,z)$ be any two distinct points in $\overline{\fH^{15}}$. Then there is an element of $G_\mO$ sending $p$ to $\infty$ and $q$ to $o=(0,0)$. In particular, the group $G_\mO$ acts doubly transitively in $\overline{\fH^{15}}$
\end{prop}

\medskip

We define a metric $d_{\fH^{15}}$ in $\fH^{15}$ as follows. For $(x,y),(w,z)\in\fH^{15}$ we set
\begin{equation}
d_{\fH^{15}}\left((x,y),(w,z)\right)=|x+\overline{w}+\overline{z}y|^{1/2}.
\end{equation}

\medskip

The proof of the statements of the next proposition is found in \cite{MP}.

\medskip

\begin{prop}\label{prop:dO-properties}
$d_{\fH^{15}}$ is a metric in $\fH^{15}$ with the following properties.
\begin{enumerate}
\item It is invariant by left translations and rotations: for each $(t,s)\in\fH^{15}$ and for each unit imaginary octonion $\mu$ we have
\begin{eqnarray*}
&&
d_{\fH^{15}}\left(T_{(t,s)}(x,y),T_{(t,s)}(w,z)\right)=d_{\fH^{15}}\left((x,y),(w,z)\right),\\
&&
d_{\fH^{15}}\left(S_\mu(x,y),S_\mu(w,z)\right)=d_{\fH^{15}}\left((x,y),(w,z)\right).
\end{eqnarray*}
\item It is scaled by a factor $\delta^2$ by dilations $D_\delta$:
\begin{equation*}
d_{\fH^{15}}\left(D_\delta(x,y),D_\delta(w,z)\right)=\delta^2d_{\fH^{15}}\left((x,y),(w,z)\right).
\end{equation*}
\item For all $(x,y),(w,z)\in\fH^{15}\setminus\{o\}$ we have
\begin{eqnarray*}
&&
d_{\fH^{15}}\left(R(x,y),o\right)=\frac{1}{d_{\fH^{15}}\left((x,y),o\right)},\\
&&
d_{\fH^{15}}\left(R(x,y),R(w,z)\right)=\frac{d_{\fH^{15}}\left((x,y),(w,z)\right)}{d_{\fH^{15}}\left((x,y),o\right)d_{\fH^{15}}\left(o,(w,z)\right)}.
\end{eqnarray*}
\end{enumerate}
\end{prop}

\medskip

It follows that the isometries of $d_{\fH^{15}}$ are left translations $T_{(t,s)}$ and transformations $S_\mu$; the subgroup of $G_\mO$ generated by these transformations will be denoted by ${\rm Isom}(\fH^{15},d_\fH)$ (in the \cite{MP} it is denoted by ${\rm Aut}(\fH^{15})$). We have ${\rm Isom}(\fH^{15},d_\fH)=\fH^{15}\times {\rm Spin}_7(\R)$ and thus it acts transitively on $\fH^{15}$. There is a natural extension of the metric $d_{\fH^{15}}$ in $\overline{\fH^{15}}$ which we shall denote again by $d_{\fH^{15}}$ and is defined by the following rules:
\begin{enumerate}
\item for each $(x,y)\in\fH^{15}$ we set $d_{\fH^{15}}\left((x,y),\infty\right)=\infty$ and
\item $d_{\fH^{15}}\left(\infty,\infty\right)=0$.
\end{enumerate}
It is obvious that the extended metric enjoys the properties of Proposition \ref{prop:dO-properties}.

\medskip

We wish to remark at this pointa that by applying the change of coordinates $(x,y)\to (\zeta,v)$ where
$$
\zeta=y/\sqrt{2},\quad v=\Im(x) 
$$
we can work in the {\it octonionic Heisenberg group} instead of the variety $\fH^{15}$, that is $\mO\times\Im(\mO)$ with group multiplication
$$
(\zeta,v)*(\zeta',v')=\left(\zeta+\zeta',v+v'+2\Im\left(\overline{\zeta}\zeta\right)\right),
$$
see also \cite{A} for a slightly different parametrization. For reasons of consistency with \cite{MP}, we prefered to work in $\fH^{15}$ instead.

\subsubsection{$\R-$circles}\label{sec:R-sirc-O}
Perhaps the easiest (and swiftest) way to define $\R-$circles in the boundary of octonionic hyperbolic plane $\partial\bH^2_\mO$ is to start from the {\it standard $\R-$ circle} $R_\R$ given by
$$
R_\R=\left\{(x,y)\in\fH^{15}\;|\;x=-t^2,\;y=\sqrt{2}t,\;\;t\in\R\right\}.
$$
Then a curve in $\partial\bH^2_\mO$ is an $\R-$circle if and only if it is the image of $R_\R$ under an element of $G_\mO$.

\section{Ptolemaean Inequality and Ptolemaeus' Theorem II}\label{sec:pt2}
In this section we prove our main result for the case where $\K=\mO$. For the proof, we use a slightly different route than the one we used in the case where $\K=\R.\C,\mH$; that is, we do not define an octonionic cross--ratio for a quadruple $\fp$ of distinct points $p_1,p_2,p_3,p_4$ in the boundary. We have mentioned in the introduction that the definition of  a  working octonionic cross--ratio seems to be quite tricky, mainly due to non associativity of the octonionic multiplication.

We circumvent this problem by defining in Section \ref{sec:real-cr-oct} the real cross--ratio. This has all the desired properties, i.e. it is invariant by the elements of $G_\mO$ and satisfies the appropriate symmetry conditions; therefore, this cross--ratio is sufficient for our purposes. Indeed, we prove Inequality \ref{eq:X-oct} in Proposition \ref{prop:X-oct} which is analogous to inequality \ref{eq-variety2}. From this inequality, the Ptolemaean inequality and the Theorem of Ptolemaeus in $\partial\bH^2_\mO$ (and subsequently in $\fH^{15}$) will follow as corollaries.

\subsection{Real cross--ratios}\label{sec:real-cr-oct}
For each quadruple of distinct points $\fp=(p_1,p_2,p_3,p_4)$ in $\overline{\fH^{15}}$ we define their real cross--ratio by
$$
\X(p_1,p_2,p_3,p_4)=\frac{d_{\fH^{15}}^2(p_4,p_2)\;d_{\fH^{15}}^2(p_3,p_1)}{d_{\fH^{15}}^2(p_4,p_1)\;d_{\fH^{15}}^2(p_3,p_2)},
$$
with the obvious modification if one of the points is $\infty$. It is evident that $\X(p_1,p_2,p_3,p_4)=$ $[p_1,p_2,p_3,p_4]$ is invariant by the action of $G$. Permuting the points we have the following symmetries:
$$
[p_1,p_2,p_3,p_4]=[p_2,p_1,p_4,p_3]=[p_3,p_4,p_1,p_2]=[p_4,p_3,p_2,p_1].
$$
Therefore, from the 24 cross--ratios corresponding to the permutations of a given quadruple, we end up with six cross--ratios which in turn are functions of the following two:
\begin{eqnarray}
&&\label{eq:X1-oct}
\X_1=[p_1,p_2,p_3,p_4]=\frac{d_{\fH^{15}}^2(p_4,p_2)\;d_{\fH^{15}}^2(p_3,p_1)}{d_{\fH^{15}}^2(p_4,p_1)\;d_{\fH^{15}}^2(p_3,p_2)},\\
&&\label{eq:X2-oct}
\X_2=[p_1,p_3,p_2,p_4]=\frac{d_{\fH^{15}}^2(p_4,p_3)\;d_{\fH^{15}}^2(p_2,p_1)}{d_{\fH^{15}}^2(p_4,p_1)\;d_{\fH^{15}}^2(p_2,p_3)}.
\end{eqnarray}
In fact, the following relations hold.
\begin{eqnarray*}
&&
[p_1,p_2,p_4,p_3]=\X_1^{-1},\\
&&
[p_1,p_3,p_4,p_2]=\X_2^{-1},\\
&&
[p_1,p_4,p_3,p_2]=\X_1\X_2^{-1},\\
&&
[p_1,p_4,p_2,p_3]=\X_2\X_1^{-1}.
\end{eqnarray*}

\medskip

We now prove a proposition analogous to Proposition \ref{prop-norm}.

\medskip

\begin{prop}\label{prop:X-oct}
Let $p_1,p_2,p_3,p_4$ be any four distinct points in ${\overline \fH^{15}}$. Let also $\X_1$ and $\X_2$ be defined by \ref{eq:X1-oct} and \ref{eq:X2-oct}. Then
\begin{equation}\label{eq:X-oct}
\X_1^2+\X_2^2-2\X_1-2\X_2-2\X_1\X_2+1\le 0.
\end{equation} 
Equality in \ref{eq:X-oct} holds if and only if all $p_i$, $i=1,\dots,4$ lie in the same $\R-$circle.
\end{prop}
\begin{proof}
Due to the double transitive action of $G$ in ${\overline \fH^{15}}$ we may normalise so that
$$
p_1=\infty,\quad p_2=(x_2,y_2),\quad p_3=(x_3,y_3),\quad p_4=o,
$$
where
$$
2\Re(x_i)+|y_i|^2=0,\quad i=2,3.
$$
We have
\begin{eqnarray*}
&&
d_{\fH^{15}}(p_1,p_2)=d_{\fH^{15}}(p_1,p_3)=d_{\fH^{15}}(p_1,p_4)=\infty,\\
&&
d_{\fH^{15}}(p_2,p_3)=|x_2+\overline{x_3}+\overline{y_3}y_2|^{1/2},\quad d_{\fH^{15}}(p_2,p_4)=|x_2|^{1/2},\\
&&
d_{\fH^{15}}(p_3,p_4)=|x_3|^{1/2},
\end{eqnarray*}
and therefore
$$
\X_1=\frac{|x_2|}{|x_3+\overline{x_2}+\overline{y_2}y_3|},\quad \X_2=\frac{|x_3|}{|x_2+\overline{x_3}+\overline{y_3}y_2|}.
$$
Set $q=\overline{x_2}+x_3+\overline{y_3}y_2$.
We calculate next
\begin{eqnarray*}
2|q|^2\X_1\X_2&=&2|x_2|\;|x_3|\\
&\ge &2\Re(x_2\overline{x_3})\\
&=&4\Re(x_2)\Re(x_3)-2\Re(x_2x_3)\\
&=&|y_2|^2|y_3|^2-2\Re(x_2x_3)\\
&=& \left|\overline{x_2}+x_3-(x_3+\overline{x_2}+\overline{y_2}y_3))\right|^2-2\Re(x_2x_3)\\
&=&|x_2|^2+|x_3|^2-2\Re(\left(x_2(x_3+\overline{x_2}+\overline{y_2}y_3)\right)\\
&&-2\Re\left(x_3(x_2+\overline{x_3}+\overline{y_3}y_2\right)+|x_3+\overline{x_2}+\overline{y_2}y_3|^2\\
&\ge &|x_2|^2+|x_3|^2-2|x_2||x_3+\overline{x_2}+\overline{y_2}y_3|-2|x_3||x_3+\overline{x_2}+\overline{y_2}y_3|+|x_3+\overline{x_2}+\overline{y_2}y_3|^2\\
&=&|q|^2\left(\X_1^2+\X_2^2-2\X_1-2\X_2+1\right),
\end{eqnarray*}
where in the third line we have used the obvious equality
$$
2\Re(x_2)\Re(x_3)=\Re(x_2x_3)+\Re\left(x_2\overline{x_3}\right).
$$
Suppose now that equality holds in Equation \ref{eq:X-oct}. This means that all intermediate inequalities in the previous proof hold as equalities, namely
\begin{equation*}
\left|x_2q\right|=\Re\left(x_2q\right),\quad \left|x_3\overline{q}\right|=\Re\left(x_3\overline{q})\right),\quad |x_2x_3|=\Re\left(x_2\overline{x_3}\right).
\end{equation*}
From these relations 
 we obtain that there exist non negative real numbers $a_1, a_2, a_3$ such that $x_2\overline{x_3}=a_1$. $x_2q=a_2$ and $x_3\overline{q}=a_3$. Observe that all $a_i$ are positive.
Indeed, if $a_1=0$, then $|x_2x_3|=|x_2|\;|x_3|=0$ which means that $\X_1\X_2=0$ which can not happen because the points $p_i$ have been considered distinct. On the other hand, if $a_2$ is zero, then $|x_2q|=|x_2|\;|q|=|x_2|d_\fH(p_2,p_3)=0$ which is again absurd and the same reasoning justifies $a_3>0$.  Therefore
\begin{equation*}\label{eq:X-oct4}
q=a_2x_2^{-1}=a_3\overline{x_3}^{-1},\quad x_2=a_1\overline{x_3}^{-1}.
\end{equation*}
Combining 
 we get
\begin{equation*}\label{eq:X-oct5}
x_2=\frac{a_1a_2}{a_3}x_2^{-1}\quad\text{and}\quad x_3=\frac{a_1a_3}{a_2}x_3^{-1}.
\end{equation*}
Taking imaginary parts in both sides of these equations we obtain the relations
$$
\Im(x_2)=-\frac{a_1a_2}{a_3}\cdot\frac{\Im(x_2)}{|x_2|^2},\quad\text{and}\quad \Im(x_3)=-\frac{a_1a_3}{a_2}\cdot\frac{\Im(x_3)}{|x_3|^2},
$$
which are absurd unless $\Im(x_2)=\Im(x_3)=0$. Therefore $x_2,x_3\in\R_-$ and moreover, 
\begin{eqnarray*}
&&
x_2=-\left(\frac{a_1a_2}{a_3}\right)^{1/2},\quad x_3=-\left(\frac{a_1a_3}{a_2}\right)^{1/2},\\
&&
|y_2|=\sqrt{2}\left(\frac{a_1a_2}{a_3}\right)^{1/4},\quad
|y_3|=\sqrt{2}\left(\frac{a_1a_3}{a_2}\right)^{1/4}.
\end{eqnarray*}
Since $x_2$ and $x_3$ are neqative, we  have
$$
|x_2q|=-x_2|q|=x_2\Re(q),
$$
thus $q\in\R_-$ and in particular $\overline{y_3}y_2=\overline{y_2}y_3=a\in\R_*$. Setting $y_2=|y_2|\mu_2$ and $y_3=|y_3|\mu_3$ where $\mu_2$ and $\mu_3$ are unit octonions we obtain
$$
\overline{\mu_3}\mu_2=\overline{\mu_2}\mu_3=\overline{\overline{\mu_3}\mu_2}
$$
and thus $\overline{\mu_3}\mu_2=\pm 1$, i.e. $\mu=\mu_2=\pm\mu_3$ and
$$
y_2=\sqrt{2}\left(\frac{a_1a_2}{a_3}\right)^{1/4}\mu,\quad y_3=\pm\sqrt{2}\left(\frac{a_1a_3}{a_2}\right)^{1/4}\mu.
$$
By applying a transformation $S_\nu$ we may assume that $\mu$ is an imaginary octonion (pick a unit imaginary octonion $\nu$ orthogonal to $\Im(\mu)$).  By applying $S_{\overline{\mu}}$ we may also assume that $y_2$ and $y_3$ are real.  
This gives in total
$$
x_3=\left(\frac{a_3}{a_2}\right)^{1/2}x_2,\quad y_3=\pm\left(\frac{a_3}{a_2}\right)^{1/4}y_2,
$$
and thus our points lie in an $\R-$circle.

We further observe that in the case of the positive sign for $y_3$ we have $\X_1^{1/2}-\X_2^{1/2}=\pm 1$ and in the case of the negative sign we have $\X_1^{1/2}+\X_2^{1/2}=1$.

The reader should now verify  easily that if all points lie in an $\R-$circle, then Inequality \ref{eq:X-oct} holds as an equality.

\end{proof}

\medskip

Since Inequality \ref{eq:X-oct} is written equivalently as
$$
\left(\X_1^{1/2}-\X_2^{1/2}-1\right)\cdot\left(\X_1^{1/2}-\X_2^{1/2}+1\right)\cdot\left(\X_1^{1/2}+\X_2^{1/2}-1\right)\cdot\left(\X_1^{1/2}+\X_2^{1/2}+1\right)\le 0,
$$
then working in the manner of the proof of the Ptolemaean Inequality \ref{thm:ptolemaean-ineq}, we obtain the Ptolemaean inequality in the boundary of the octonionic hyperbolic plane as a corollary to Proposition \ref{prop:X-oct}.

\medskip

\begin{thm}\label{thm:ptolemaean-ineq-oct}{\bf (Ptolemaean inequality  in $\partial\bH^2_\mO)$}
Let $\fp=(p_1,p_2,$ $p_3,p_4)$ a quadruple of distinct points in $\partial \bH_\mO^2$ and $\X_i=\X_i(\fp)$, $i=1,2$  its corresponding real  cross--ratios defined in \ref{eq:X1-oct} and \ref{eq:X2-oct} respectively.  Then  the following inequalities hold:
 \begin{equation*}
  \X_1^{1/2}+\X_2^{1/2}\ge 1 \quad \text{and}\quad -1\le\X_1^{1/2}-\X_2^{1/2}\le 1.
 \end{equation*}

\end{thm}

Ptolemaeus' Theorem then follows directly from the necessary and sufficient condition for equality in Inequality \ref{eq:X-oct}.

\medskip 

\begin{thm}{(\bf Ptolemaeus' Theorem in $\partial\bH^2_\mO$)}\label{thm:ptolemaean-eq-oct}
 Each inequality  \ref{thm:ptolemaean-ineq-oct} holds if and only if all four points of $\fp$ lie in an $\R-$circle. Then, 
\begin{enumerate}
 \item $\X_1^{1/2}-\X_2^{1/2}=1$ if $p_1$ and $p_3$ separate $p_2$ and $p_4$;
 \item $\X_2^{1/2}-\X_1^{1/2}=1$ if  $p_1$ and $p_2$ separate $p_3$ and $p_4$;
 \item $\X_1^{1/2}+\X_2^{1/2}=1$ if  $p_1$ and $p_4$ separate $p_2$ and $p_3$.
\end{enumerate}
\end{thm}

\medskip

We finally state

\medskip

\begin{thm}\label{thm:ptolemaean-heis-oct} {\bf (Ptolemaean inequality and  Ptolemaeus' Theorem in the variety $\fH^{15}$)}
 The  metric $d_\fH$ in  $\fH^{15}$ satisfies the Ptolemaean inequality: for each quadruple of points $\fp=(p_1,p_2,p_3,p_4)$, inequalities \ref{eq:ineq-heis1}, \ref{eq:ineq-heis2} and \ref{eq:ineq-heis3} hold. Moreover, each of these inequalities hold as an equality if and only if all points lie in an $\R-$circle. Explicitly,
\begin{enumerate}
\item $d_\fH(p_2,p_3)\cdot d_\fH(p_1,p_4)=d_\fH(p_2,p_4)\cdot d_\fH(p_1,p_3)+d_\fH(p_1,p_2)\cdot d_\fH(p_3,p_4)$ if and only if all points lie in an $\R-$circle and $p_1$ and $p_4$ separate $p_2$ and $p_3$;
\item $d_\fH(p_1,p_3)\cdot d_\fH(p_2,p_4)=d_\fH(p_1,p_2)\cdot d_\fH(p_3,p_4)+d_\fH(p_2,p_3)\cdot d_\fH(p_1,p_4)$ if and only if all points lie in an $\R-$circle and $p_1$ and $p_3$ separate $p_2$ and $p_4$;
\item $d_\fH(p_1,p_2)\cdot d_\fH(p_3,p_4)=d_\fH(p_1,p_3)\cdot d_\fH(p_2,p_4)+d_\fH(p_2,p_3)\cdot d_\fH(p_1,p_4)$ if and only if all points lie in an $\R-$circle and $p_1$ and $p_2$ separate $p_3$ and $p_4$.
\end{enumerate}
\end{thm}

\end{document}